\documentclass[11pt,reqno]{amsart}
\usepackage{amsmath}
\usepackage{amssymb}
\usepackage{amsthm}
\usepackage{color}
\usepackage{eucal}
\usepackage{tikz}
\usepackage{gastex}

\address[karl.auinger@univie.ac.at]{K.~Auinger, Fakult\"at f\"ur Mathematik, Universit\"at Wien, Nordbergstrasse 15, A-1090 Wien, Austria}

\newtheorem{theorem}{Theorem}[section]
\newtheorem{corollary}[theorem]{Corollary}
\newtheorem{lemma}[theorem]{Lemma}
\newtheorem{proposition}[theorem]{Proposition}

\theoremstyle{definition}

\DeclareMathOperator{\rk}{rk}
\DeclareMathOperator{\dom}{dom} \DeclareMathOperator{\ran}{ran}
\DeclareMathOperator{\Sing}{Sing}
\def\softl{l\kern-0.3ex\raise0.1ex\hbox{'}\kern-0.3ex}   
\DeclareSymbolFont{rsfscript}{OMS}{rsfs}{m}{n}
\DeclareSymbolFontAlphabet{\mathrsfs}{rsfscript}
\def\rrbb{{-\!\!\!-\!\!\!-\!\!\!-\!\!\!-}}
\def\ep{\varepsilon}
\def\Dc{\mathrsfs{D}}
\def\Hc{\mathrsfs{H}}

\def\Sim{\mathfrak{S}}
\def\si{\sigma}
\def\wt{\widetilde}
\def\circledm{\protect\mathbin{\hbox
    {\protect$\bigcirc$\rlap{\kern-8.1pt\raise0.7pt\hbox
    {\protect$\mathtt{m}$}}}}}        
\def\malc{\circledm}            

\title[Krohn--Rhodes complexity]{Krohn--Rhodes complexity\\ of Brauer type semigroups}

\author{K.~Auinger}
\address{Fakult\"at f\"ur Mathematik, Universit\"at Wien,
Nordbergstrasse 15, A-1090 Wien, Austria}
\email{karl.auinger@univie.ac.at}

\begin{document}

\begin{abstract}
The Krohn--Rhodes complexity of the Brauer type semigroups $\mathfrak{B}_n$ and $\mathfrak{A}_n$ is computed. In three-quarters of the cases the result is the `expected' one: the complexity coincides with the (essential) $\mathrsfs{J}$-depth of the respective semigroup. The exception (and perhaps the most interesting case) is the annular semigroup $\mathfrak{A}_{2n}$  of even degree in which case the complexity is the  $\mathrsfs{J}$-depth minus $1$. For the `rook' versions $P\mathfrak{B}_n$ and $P\mathfrak{A}_n$ it is shown that $c(P\mathfrak{B}_n)=c(\mathfrak{B}_n)$ and $c(P\mathfrak{A}_{2n-1})=c(\mathfrak{A}_{2n-1})$ for all $n\ge 1$. The computation of $c(P\mathfrak{A}_{2n})$ is left as an open problem.
\end{abstract}

\maketitle

\section{Introduction and Background}
It follows from the famous Krohn--Rhodes Prime Decomposition Theorem \cite{krohnrhodes} that  each finite semigroup $S$ divides an iterated  wreath product $$A_n\wr G_n\wr A_{n-1}\wr \cdots \wr A_1\wr G_1\wr A_0$$
were the $G_i$ are groups and the $A_i$ are aperiodic semigroups. The number $n$ of non-trivial group components of the shortest such iterated product is the \emph{group complexity} or \emph{Krohn--Rhodes complexity} of the semigroup $S$. The question whether this number is algorithmically computable given the semigroup $S$ as input is perhaps \textbf{the} most fruitful research problem in finite semigroup theory. To the author's knowledge, this problem is still open despite the tremendous effort that has been spent on it over the years.

Concerning \emph{classes} of abstract semigroups, the pseudovariety $L\mathbf{G}\malc\mathbf{A}$ is the largest one which contains semigroups of arbitrarily high complexity and for which at present an algorithm is known which computes the complexity of each member --- this includes $\mathbf{DS}$ and thus completely regular semigroups (unions of groups). Another result is that the ``complexity-$\frac{1}{2}$'' pseudovarieties $\mathbf{A}*\mathbf{G}$ and $\mathbf{G}*\mathbf{A}$ have decidable membership (the latter being contained in $L\mathbf{G}\malc\mathbf{A}$).  On the other hand, the complexity of many naturally occurring \emph{individual} and concrete semigroups is known: these include the  semigroup of all transformations of a finite set and the semigroup of all endomorphisms of a finite vector-space \cite{rhodesTn}, as well as the semigroup of all binary relations on a finite set \cite{rhodesBn}. More recently, Kambites \cite{kambites} calculated the complexity of the semigroup of all upper triangular matrices over a finite field. The present paper intends to contribute to the latter kind of results. Indeed, we shall present a calculation of the complexity of the Brauer  semigroup $\mathfrak{B}_n$ and the annular semigroup $\mathfrak{A}_n$ (these occur originally in representation theory of associative algebras but have recently attracted considerable attention among semigroup theorists). It turns out that the cases of $\mathfrak{B}_n$ and $\mathfrak{A}_{2n+1}$ can be treated in a straightforward fashion by the use of arguments that apply to transformation semigroups and linear semigroups. The case of the annular semigroup $\mathfrak{A}_{2n}$ of even degree is somehow different. Although the problem can be solved by  use of the machinery developed by the Rhodes school, the solution requires quite a bit of care and is much less obvious. Actually, the author was not able to compute the complexity $c(\mathfrak{A}_{2n})$ directly. The strategy is rather to look at a certain natural subsemigroup $\mathfrak{EA}_{2n}$ first and calculate the complexity of this subsemigroup. In a second step it is then shown that the complexity of the full semigroup $\mathfrak{A}_{2n}$ does not exceed the complexity of the `even' subsemigroup $\mathfrak{EA}_{2n}$. It should be mentioned that none of the semigroups $\mathfrak{B}_n$ or $\mathfrak{A}_n$ is contained in $L\mathbf{G}\malc\mathbf{A}$ except for $n=1$.

The paper is organized as follows. In Section \ref{prelim} we collect all preliminaries on Brauer and annular semigroups as well as the basics of Krohn--Rhodes complexity needed in the sequel. In Section \ref{Brauer} the complexity of the Brauer semigroup is computed to be $c(\mathfrak{B}_n)=\lfloor\frac{n}{2}\rfloor$ which is exactly the (essential) $\mathrsfs{J}$-depth. It is also shown that the partial Brauer semigroup $P\mathfrak{B}_n$ has the same complexity as the `total' counterpart $\mathfrak{B}_n$. In Section \ref{annular} we first treat the annular semigroup of even degree and show that $c(\mathfrak{A}_{n})=\frac{n}{2}-1$. This is the difficult case and is treated with the help of a certain subsemigroup, the \emph{even annular semigroup} $\mathfrak{EA}_n$. Afterwards the odd degree case is treated which is again, in a sense, standard. Finally, some remarks on  the partial versions $P\mathfrak{A}_n$ are given. In the odd case, the complexity is the same as that of their total counterparts while the computation of the complexity in the even case is left as an open problem.

Throughout the paper, all semigroups are assumed to be finite. For background information on (finite) semigroups the reader is referred to the monographs by Almeida \cite{almeidauniv} and Rhodes and Steinberg \cite{qtheory}.
\section{Preliminaries}\label{prelim}
\subsection{Brauer type semigroups} Here we present the basic definitions and results concerning Brauer type semigroups. For each positive integer $n$ we are going to define:
\begin{itemize}
\item the \emph{partition semigroup} $\mathfrak{C}_n$,
\item the \emph{Brauer semigroup} $\mathfrak{B}_n$,
\item the \emph{partial} Brauer semigroup $P\mathfrak{B}_n$,
\item the \emph{Jones semigroup} $\mathfrak{J}_n$,
\item the \emph{annular semigroup} $\mathfrak{A}_n$,
\item the \emph{partial} annular semigroup $P\mathfrak{A}_n$.
\end{itemize}
The semigroups $\mathfrak{C}_n$, $\mathfrak{B}_n$, $\mathfrak{A}_n$ and $\mathfrak{J}_n$ arise as vector space bases of certain associative algebras which are
relevant in representation theory \cite{brauer, halversonram, jonesannular, grahamlehrer}. The semigroup structure and related questions
for the above-mentioned semigroups have been studied  by several authors, see, for example, \cite{aubrauer, adv2, kudmaz, malmaz, Maz1, Maz2}.

We start with the definition of $\mathfrak{C}_n$. For each positive integer $n$ let
$$[n]=\{1,\dots,n\},\quad [n]'=\{1',\dots,n'\},\quad [n]''=\{1'',\dots,n''\}$$
be three pairwise disjoint copies of the set of the first $n$ positive integers and put
$$\widetilde{[n]}=[n]\cup [n]'.$$
The base set of the partition semigroup $\mathfrak{C}_n$ is the set of all partitions of the set $\widetilde{[n]}$; throughout, we
consider a partition of a set and the corresponding equivalence relation on that set as two different views of the same
thing and without further mention we freely switch between these views, whenever it seems to be convenient. For
$\xi,\eta\in \mathfrak{C}_n$, the product $\xi\eta$ is defined (and computed) in four steps \cite{Wil}:
\begin{enumerate}
\item Consider the $'$-analogue of $\eta$: that is, define $\eta'$
on ${[n]'}\cup {[n]''}$ by
$${x'}\mathrel{\eta'}{y'}:\Leftrightarrow x\mathrel{\eta}
y\text{ for all } x,y\in \widetilde{[n]}.$$
\item Let $\left<\xi,\eta\right>$ be the equivalence relation on
$\widetilde{[n]}\cup {[n]''}$ generated by $\xi\cup {\eta'}$, that is, set $\left<\xi,\eta\right>:=(\xi\cup {\eta'})^t$
where $^t$ denotes the transitive closure.
\item Forget all elements having a single prime $'$: that is, set
$$\left<\xi,\eta\right>^\circ:=\left<\xi,\eta\right>|_{[n]\cup{[n]''}}.$$
\item Replace  double primes with single primes to
obtain the product $\xi\eta$: that is, set
$$x\mathrel{\xi\eta}y:\Leftrightarrow f(x)\mathrel{\left<\xi,\eta\right>^\circ}f(y)
\text{ for all }x,y\in \widetilde{[n]}$$ where $f:\widetilde{[n]}\to [n]\cup{[n]''}$ is the bijection
$$x\mapsto x, x'\mapsto x'' \text{ for all } x\in [n].$$
\end{enumerate}

For example, let $n=5$ and
$$\xi\rule[-22.5mm]{0pt}{45mm}=
\begin{picture}(40,0)(0,22.5)
\unitlength=0.8mm \drawrect(0,15,9,54) \put(3,8){5} \put(3,18){4} \put(3,28){3} \put(3,38){2} \put(3,48){1}
\drawrect(37,5,46,14) \put(40,8){$5'$} \drawrect(37,15,46,24) \put(40,18){$4'$} \put(40,28){$3'$}
\drawrect(37,35,46,44) \put(40,38){$2'$} \put(40,48){$1'$}
\drawline[AHnb=0](0,5)(0,14)(25,14)(25,54)(46,54)(46,45)(34,45)(34,34)(46,34)(46,25)(34,25)(34,5)(0,5)
\end{picture}\ ,\qquad
\eta=
\begin{picture}(40,0)(0,22.5)
\unitlength=0.8mm \drawrect(0,5,9,14) \put(3,8){5} \put(3,18){4} \put(3,28){3} \put(3,38){2} \put(3,48){1}
\drawrect(37,5,46,14) \put(40,8){$5'$} \drawrect(0,15,46,24) \put(40,18){$4'$} \put(40,28){$3'$} \put(40,38){$2'$}
\put(40,48){$1'$} \drawline[AHnb=0](0,45)(0,54)(46,54)(46,35)(37,35)(37,45)(0,45)
\drawline[AHnb=0](0,25)(0,44)(9,44)(9,34)(46,34)(46,25)(0,25)
\end{picture}\ .
$$
Then
$$\langle\xi,\eta\rangle\rule[-22.5mm]{0pt}{50mm}=
\begin{picture}(80,0)(0,22.5)
\unitlength=0.8mm \drawrect(0,15,9,54) \put(3,8){5} \put(3,18){4} \put(3,28){3}
\put(3,38){2} \put(3,48){1} \drawrect(37,5,46,14) \put(40,8){$5'$} \drawrect(37,15,83,24) \put(40,18){$4'$}
\put(40,28){$3'$} \put(40,38){$2'$} \put(40,48){$1'$} \drawrect(74,5,83,14) \put(77,8){$5''$} \put(77,18){$4''$}
\put(77,28){$3''$} \put(77,38){$2''$} \put(77,48){$1''$}
\drawline[AHnb=0](0,5)(0,14)(25,14)(25,54)(83,54)(83,25)(34,25)(34,5)(0,5)
\end{picture}
$$
and
$$\xi\eta\rule[-22.5mm]{0pt}{50mm}=
\begin{picture}(40,0)(0,22.5)
\unitlength=0.8mm \drawrect(0,15,9,54) \put(3,8){5} \put(3,18){4} \put(3,28){3} \put(3,38){2} \put(3,48){1}
\drawrect(37,5,46,14) \put(40,8){$5'$} \drawrect(37,15,46,24) \put(40,18){$4'$} \put(40,28){$3'$} \put(40,38){$2'$}
\put(40,48){$1'$} \drawline[AHnb=0](0,5)(0,14)(25,14)(25,54)(46,54)(46,25)(34,25)(34,5)(0,5)
\end{picture}\ .$$

This multiplication is associative making $\mathfrak{C}_n$ a semigroup with identity $1$ where
$$1=\{\{k,k'\}\mid k\in [n]\}.$$
The group of units of $\mathfrak{C}_n$ is the symmetric group $\Sim_n$ (acting on $[n]$ on the right) with canonical embedding
$\Sim_n\hookrightarrow \mathfrak{C}_n$ given by
$$\sigma\mapsto \{\{k,(k\si)'\}\mid k\in [n]\} \text{ for all } \si\in
\Sim_n.$$ More generally, the semigroup of all (total) transformations $\mathfrak{T}_n$ of $[n]$ acting on the right is also naturally
embedded in $\mathfrak{C}_n$ by
\begin{equation}
\label{partialtransformations} \phi\mapsto \{\{k'\}\cup
k\phi^{-1}\mid k\in [n]\}.
\end{equation}
If $k$ is not in the image of $\phi$ then $\{k'\}$ forms by definition a singleton class. The equivalence classes of
some $\xi\in \mathfrak{C}_n$ are usually referred to as \emph{blocks}; the \emph{rank} $\rk\xi$ is the number of blocks of $\xi$
whose intersection with $[n]$ as well as with $[n]'$ is not empty --- this coincides with the usual notion of rank of a
mapping on $[n]$ in case $\xi$ is in the image of the embedding \eqref{partialtransformations}. It is  known that the
rank characterizes the $\Dc$-relation in $\mathfrak{C}_n$ \cite{Maz1,kudmaz}: for any $\xi,\eta\in\mathfrak{C}_n$, one has
$\xi\mathrel{\Dc}\eta$ if and only if $\rk\xi=\rk\eta$.

The semigroup $\mathfrak{C}_n$ admits a natural inverse involution making it a regular $*$-semi\-group: consider first the permutation $^*$
on $\widetilde{[n]}$ that swaps primed with unprimed elements, that is, set
\[
k^*=k',\ (k')^*=k\text{ for all } k\in [n].
\]
Then define, for $\xi\in \mathfrak{C}_n$,
\[
x\mathrel{\xi^*}y:\Leftrightarrow {x^*}\mathrel{\xi}{y^*} \text{ for all }x,y\in \widetilde{[n]}.
\]
That is, $\xi^*$ is obtained from $\xi$ by interchanging in $\xi$ the primed with the unprimed elements. It is easy to
see that
\begin{equation}
\label{projections} \xi^{**}=\xi,\ (\xi\eta)^*=\eta^*\xi^*\ \text{ and }\ \xi\xi^*\xi=\xi\ \text{ for all }\
\xi,\eta\in \mathfrak{C}_n.
\end{equation}
The elements of the form $\xi\xi^*$ are called \emph{projections}. They are idempotents (as one readily sees from the
last equality in \eqref{projections}).
We note that in the group $\Hc$-class of any projection, the involution $^*$ coincides with the inverse operation in
that group.

The \emph{Brauer semigroup} $\mathfrak{B}_n$ can
be conveniently defined as a subsemigroup of $\mathfrak{C}_n$: namely, $\mathfrak{B}_n$
 consists of all elements of $\mathfrak{C}_n$ all of
whose blocks have size $2$; the \emph{partial} Brauer semigroup $P\mathfrak{B}_n$ consists of all elements of $\mathfrak{C}_n$ all of whose blocks have size at most~2. It is useful to think of the elements
of $P\mathfrak{B}_n$ respectively $\mathfrak{B}_n$ in terms of \emph{diagrams}. These are pictures like the one in Figure \ref{membersofBn}.
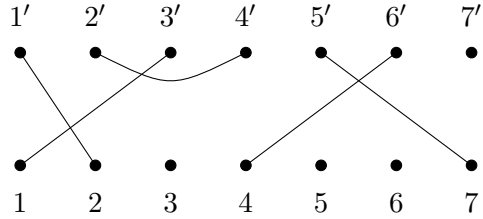
\begin{figure} [ht]
\centering
  \begin{tikzpicture}
[scale=0.5]
\foreach \x in {0,2,4,6,8,10,12} \foreach \y in {0,3} \filldraw (\x,\y) circle (4pt);
\draw (0,-1) node {$1$}; \draw (2,-1) node {$2$};\draw (4,-1) node {$3$};\draw (6,-1) node {$4$};
\draw (8,-1) node {$5$};\draw (10,-1) node {$6$};\draw (12,-1) node {$7$};
\draw (0,4) node {$1'$};\draw (2,4) node {$2'$};\draw (4,4) node {$3'$};\draw (6,4) node {$4'$};
\draw (8,4) node {$5'$};\draw (10,4) node {$6'$};\draw (12,4) node {$7'$};
\draw (0,0) -- (4,3); \draw (0,3) -- (2,0); \draw (6,0) -- (10,3); \draw (8,3) -- (12,0); 
\draw (2,3) .. controls (4,2) and (4,2) .. (6,3);
\end{tikzpicture}
\caption[diagram]{A diagram in $P\mathfrak{B}_7$}\label{membersofBn}
\end{figure}

  Both
semigroups $\mathfrak{B}_n$ and $P\mathfrak{B}_n$ are closed under $^*$. In both types of semigroups, the group $\Hc$-class of
a projection $\xi\xi^*$ of rank $t$ is isomorphic (as a regular $*$-semigroup) with the symmetric group $\Sim_t$. Let
$\xi\in\mathfrak{B}_n$ be of rank $t$ and let
$$\{k_1,l'_1\},\dots,\{k_t,l'_t\},\ \mbox{for some } k_i,l_i\in [n],$$
be the blocks of $\xi$ which contain an element of $[n]$ and of $[n]'$.
Then $\{k_1,\dots, k_t\}$ and $\{l'_1,\dots,l'_t\}$ is the \emph{domain} $\dom \xi$ respectively \emph{range}
$\ran \xi$ of $\xi$. For any projection $\ep$ we obviously have $\ran \ep=(\dom \ep)'$.

The \emph{Jones semigroup} (also called \emph{Temperley--Lieb semigroup}, see also \cite{temperleylieb})\footnote{Following \cite{FitzGerald}, we use the term \emph{Jones semigroup}.} $\mathfrak{J}_n$ is the subsemigroup of $\mathfrak{B}_n$ consisting of all diagrams that can be drawn in the plane within a rectangle (as in Figure \ref{membersofBn}) in a way such that any two of its lines have empty intersection. These diagrams are called \emph{planar}. It is well-known  and easy to see  that $\mathfrak{J}_n$ is aperiodic \cite{FitzGerald}.

Next we define the annular semigroup $\mathfrak{A}_n$ \cite{jonesannular} and the partial annular semigroup $P\mathfrak{A}_n$. These will also be realized as  certain subsemigroups of the (partial) Brauer
semigroup. For this purpose it is convenient to first represent the elements of $P\mathfrak{B}_n$ (and therefore of $\mathfrak{B}_n$) as \emph{annular diagrams}.
Consider an annulus $A$ in the complex plane, say $A=\{z\mid 1<|z|<2\}$ and identify the elements of $\wt{[n]}$ with
certain points of the boundary of $A$ via
$$k\mapsto e^{\frac{2\pi i(k-1)}{n}}\ \text{ and }\ k'\mapsto
2e^{\frac{2\pi i(k-1)}{n}}\ \text{ for all }\ k\in [n].$$ For $\xi\in P\mathfrak{B}_n$ (in particular, for $\xi\in \mathfrak{B}_n$) take a copy of $A$ and link any
$x,y\in\wt{[n]}$ with $x\ne y$ and $\{x,y\}\in \xi$ by a path (called \emph{string}) running entirely in $A$ (except for its
endpoints). For example, the element $\xi\in P\mathfrak{B}_4$ given by
$$\xi=\{\{1\},\{1'\},\{2',4'\},\{2,3'\},\{3,4\}\}$$
can then be represented by the annular diagram in Figure \ref{diagram}.
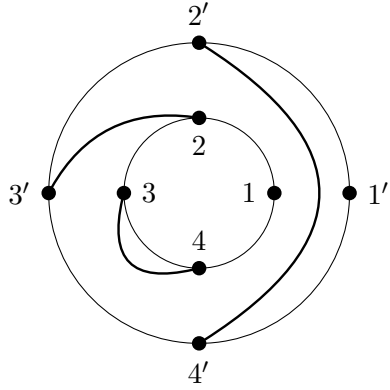
\begin{figure}[ht]
\centering
\unitlength=0.8mm
\begin{picture}(50,64)(0,-32)
\drawcircle(25,0,25) \drawcircle(25,0,50) \gasset{AHnb=0,ExtNL=y,Nfill=y,Nw=2,Nh=2,Nmr=1,linewidth=.4}
\node[NLangle=-90,NLdist=2](D1)(25,-25){$4'$} \node[NLangle=180,NLdist=2](C1)(0,0){$3'$} \node[NLdist=2](B1)(25,25){$2'$}
\node[NLangle=0,NLdist=2](A1)(50,0){$1'$} \node[NLdist=2](D2)(25,-12.5){$4$} \node[NLangle=0,NLdist=2](C2)(12.5,0){$3$}
\node[NLangle=-90,NLdist=2](B2)(25,12.5){$2$} \node[NLangle=180,NLdist=2](A2)(37.5,0){$1$} 
\drawedge[curvedepth=-20](D1,B1){} \drawedge[curvedepth=5](C1,B2){} \drawedge[curvedepth=7.5](D2,C2){}
\end{picture}
\caption{Annular diagram representation of
 a member of $P\mathfrak{A}_4$}\label{diagram}
\end{figure}
Paths representing blocks of the form $\{x,y'\}$ [$\{x,y\}$ and
$\{x',y'\}$, respectively] for some $x,y\in[n]$ are called
\emph{through strings} [\emph{inner} and \emph{outer strings},
respectively]. The \emph{annular semigroup} $\mathfrak{A}_n$ by definition
consists of all elements of $\mathfrak{B}_n$ that have a representation as an
annular diagram any two of whose strings have empty intersection. One can compose annular diagrams in an obvious way, modelling
the multiplication in $\mathfrak{B}_n$ --- from this it follows  that $\mathfrak{A}_n$
is closed under the multiplication of $\mathfrak{B}_n$. Clearly, $\mathfrak{A}_n$ is
closed under $^*$, as well.
Analogously, one gets the \emph{partial} annular semigroup $P\mathfrak{A}_n$  by
considering all elements of $P\mathfrak{B}_n$ which admit a representation by an annular diagram in which any two distinct strings have empty intersection. Again each $P\mathfrak{A}_n$ is closed under $^*$.

The notions of ``planar diagram'' and ``annular diagram'' make sense also for the elements of $\mathfrak{C}_n$; one can define the \emph{planar monoid} $\mathfrak{P}_n$ consisting of all members of $\mathfrak{C}_n$ that admit a representation as a planar diagram in which (the prepresentation of) any two distinct blocks have empty intersection (for example, the elements $\xi$ and $\eta$ in the example after the definition of the multiplication in $\mathfrak{C}_n$ belong to $\mathfrak{P}_5$), see \cite{aubrauer,halversonram}. Similarly, one could define the \emph{planar annular monoid} $\mathfrak{P}\mathfrak{A}_n$, consisting of all members of $\mathfrak{C}_n$ that admit a representation as an annular diagram in which (the representation of) any two distinct blocks have empty intersection. However, from our point of view, this gives nothing new: $\mathfrak{P}_n$ is known to be isomorphic with $\mathfrak{J}_{2n}$ for each $n$ \cite{aubrauer,halversonram} while $\mathfrak{P}\mathfrak{A}_n$ can be shown to be isomorphic with the even annular monoid $\mathfrak{EA}_{2n}$ (to be defined below) for each $n$.

 Finally, we fix the following notation: if the semigroup $\mathfrak{M}$ happens to be a monoid then its group of units is denoted by $\mathfrak{M}^\times$ while the \emph{singular part of $\mathfrak{M}$}, that is, the subsemigroup  of all non-invertible elements $\mathfrak{M}\setminus \mathfrak{M}^\times$ is denoted by $\Sing \mathfrak{M}$.

\subsection{Krohn--Rhodes complexity} Here we present the basics of Krohn--Rhodes complexity needed in the sequel. A comprehensive treatment of the subject can be found in Part II of the monograph \cite{qtheory}. Throughout, the complexity of a semigroup $S$ is denoted by $c(S)$. A $\mathrsfs{J}$-class of a semigroup is \emph{essential} if it contains a non-trivial subgroup. The \emph{depth} of a semigroup is the length $n$ of the longest chain $J_1 > J_2>\cdots > J_n$ of essential $\mathrsfs{J}$-classes. The complexity $c(S)$ of a semigroup $S$ can never exceed its depth \cite[Theorem 4.9.15]{qtheory}.
\begin{lemma}\label{ideal lemma} For each semigroup $S$ and for each ideal $I$ of $S$ the inequality $c(S)\le c(I)+ c(S/I)$ holds.
\end{lemma}
The latter statement is usually known as the \emph{Ideal Theorem} \cite[Theorem 4.9.17]{qtheory}. The next result (due to Allen and Rhodes) can be also found as Proposition 4.12.20 in \cite{qtheory}.
\begin{lemma}\label{submonoids} Let $S$ be a semigroup and let $e$ be an idempotent of $S$; then $c(SeS)=c(eSe)$.
\end{lemma}
The following  result \cite[Proposition 4.12.23]{qtheory} is useful for computing the complexity of the full transformation semigroup $\mathfrak{T}_n$ and the full linear semigroup $\mathfrak{M}_n(\mathbb{F}_q)$ over a finite field. In our situation it is helpful for the Brauer semigroup $\mathfrak{B}_n$ and the annular semigroup $\mathfrak{A}_{2n+1}$ of odd degree. For any semigroup $S$, we denote by $E(S)$ the set of all idempotents of $S$.

\begin{proposition}\label{principal proposition} Suppose that $S$ is a monoid with non-trivial group of units $G$ such that $S=\left<G, e\right>$ for some idempotent $e\notin G$ and $SeS\subseteq \left<E(S)\right>$. Then $c(S)=c(eSe)+1$.
\end{proposition}
According to \cite[Definition 4.12.11]{qtheory} a semigroup $S$ is a $\mathcal{T}_1$\emph{-semigroup} if there exists an $\mathrsfs{L}$-chain $s_1\mathrel{{\le}_{\mathrsfs{L}}}s_2\mathrel{{\le}_{\mathrsfs{L}}}\cdots\mathrel{{\le}_{\mathrsfs{L}}}s_n$ of elements of $S$ such that $S$ is generated by $s_1,\dots,s_n$. The \emph{type II subsemigroup} $\mathsf{K}_{\mathbf G}(S)$ of a semigroup $S$ consists of all elements of $S$ that relate to $1$ under every relational morphism from $S$ to a group $G$. Ash's famous theorem \cite{ash} (verifying the Rhodes type II conjecture) states that $\mathsf{K}_{\mathbf G}(S)$ is the smallest subsemigroup of $S$ that contains all idempotents and is closed under weak conjugation. The combination of Theorems 4.12.14 and 4.12.8 in \cite{qtheory} yields:
\begin{proposition}\label{kernelcomplexity} For each $\mathcal{T}_1$-semigroup $S$ which is not aperiodic, the inequality $c(\mathsf{K}_{\mathbf G}(S))<c(S)$ holds.
\end{proposition}
The final preliminary result presents the well known characterization of the members of the pseudovariety $\mathbf{A}*\mathbf{G}$. Indeed, by the definition of the Ma\softl cev Product \cite{qtheory}, a semigroup $S$ belongs to $\mathbf{A}\malc\mathbf{G}$ if and only if $\mathsf{K}_{\mathbf G}(S)$ belongs to $\mathbf{A}$; from $\mathbf{A}*\mathbf{G}=\mathbf{A}\malc\mathbf{G}$ \cite[Theorem 4.8.4]{qtheory} we get:
\begin{proposition}\label{A*G} A semigroup $S$ belongs to $\mathbf{A}*\mathbf{G}$ if and only if $\mathsf{K}_{\mathbf G}(S)$ belongs to $\mathbf{A}$.
\end{proposition}
\section{The Brauer semigroup $\mathfrak{B}_n$}\label{Brauer}
It should be mentioned that the full partition semigroup $\mathfrak{C}_n$ has complexity $n-1$ for each $n$. Indeed, it has $n-1$ essential $\mathrsfs{J}$-classes hence $c(\mathfrak{C}_n)$ can be at most $n-1$. On the other hand, the full transformation semigroup $\mathfrak{T}_n$ on $n$ letters embeds into $\mathfrak{C}_n$ and it is a classical result \cite[Theorem 4.12.31]{qtheory} that $c(\mathfrak{T}_n)=n-1$. So $c(\mathfrak{C}_n)$ has to be at least $n-1$. Of course, the Jones semigroup $\mathfrak{J}_n$ has complexity $0$ for each $n$.

Let us next consider the Brauer semigroup $\mathfrak{B}_n$. Note that $\mathfrak{B}_{2n}$ as well as $\mathfrak{B}_{2n+1}$ have $n$ essential $\mathrsfs{J}$-classes. For each pair $i<j$ with $i,j\in [n]$ define the diagram $\gamma_{ij}$ as follows:
\begin{equation}\label{atoms}
\gamma_{ij}:=\{\{i,j\},\{i',j'\},\{k,k'\}\mid k\ne i,j\}.
\end{equation}
Each $\gamma_{ij}$ is a projection of rank $n-2$. Proposition 2 in \cite{malmaz} tells us that the singular part of $\mathfrak{B}_n$ is generated by the projections $\gamma_{ij}$:
\begin{equation}\label{singularBn}\Sing \mathfrak{B}_n=\left<\gamma_{ij}\mid 1\le i<j\le n\right>.\end{equation}
Recall that the group of units of $\mathfrak{B}_n$ is the symmetric group on $n$ letters, denoted $\mathfrak{S}_n$. Another result to be essential is that $\mathfrak{B}_n$ is generated by its group of units together with $\gamma_{12}$ --- see the first paragraph of Section 3 in \cite{kudmaz} (in fact, for every $i<j$, $\gamma_{ij}$ can be used here instead of $\gamma_{12}$):
\begin{equation}\label{GegeneratedBn}\mathfrak{B}_n=\left<\mathfrak{S}_n,\gamma_{12}\right>.\end{equation}

Then $\mathfrak{B}_n\gamma_{12}\mathfrak{B}_n=\Sing \mathfrak{B}_n$ holds, and so we have $\mathfrak{B}_n\gamma_{12}\mathfrak{B}_n\subseteq \left<E(\mathfrak{B}_n)\right>$ by (\ref{singularBn}). Therefore, since $\mathfrak{B}_{n-2}\cong \gamma_{12}\mathfrak{B}_n\gamma_{12}$,
 Proposition \ref{principal proposition} implies:
\begin{proposition} The equality $c(\mathfrak{B}_n)=c(\mathfrak{B}_{n-2})+1$ holds for each $n\ge 3$.
\end{proposition}
Taking into account that $c(\mathfrak{B}_1)=0$ and $c(\mathfrak{B}_2)=1$ we obtain already the main result of this section:
\begin{theorem}\label{2n=2n+1} The equality $c(\mathfrak{B}_{2n})=c(\mathfrak{B}_{2n+1})=n$ holds for each positive integer $n$.
\end{theorem}
For the partial analogue $P\mathfrak{B}_n$ let $n\ge 2$ and denote by $P\mathfrak{B}_n^{(n-2)}$ the ideal of $P\mathfrak{B}_n$ consisting of all elements of rank at most $n-2$. The Rees quotient $P\mathfrak{B}_n/P\mathfrak{B}_n^{(n-2)}$ is an inverse semigroup and therefore has complexity $1$ (see, for example, \cite[Cor.~4.1.8]{qtheory}).   Since $P\mathfrak{B}_n^{(n-2)}=P\mathfrak{B}_n\gamma_{12}P\mathfrak{B}_n$ and $\gamma_{12}P\mathfrak{B}_n\gamma_{12}\cong P\mathfrak{B}_{n-2}$ and by use of the Ideal Theorem (Lemma \ref{ideal lemma}) and Lemma \ref{submonoids} it follows that $c(P\mathfrak{B}_n)\le c(P\mathfrak{B}_{n-2})+1$ holds for each $n\ge 3$. In other words, the transition from $P\mathfrak{B}_{n-2}$ to $P\mathfrak{B}_n$ increments the complexity  by at most $1$. Since $c(P\mathfrak{B}_1)=0$ and $c(P\mathfrak{B}_2)=1$ (the former is aperiodic, the latter has only one essential $\mathrsfs{J}$-class) it follows by induction that $c(P\mathfrak{B}_{2n})\le n$ and $c(P\mathfrak{B}_{2n+1})\le n$ for all $n$. On the other hand, since $c(P\mathfrak{B}_n)\ge c(\mathfrak{B}_n)$ the reverse inequalities also hold by Theorem \ref{2n=2n+1}. Altogether we have proved:
\begin{corollary} The equality $c(P\mathfrak{B}_n)=c(\mathfrak{B}_n)=\lfloor\frac{n}{2}\rfloor$ holds for each positive integer $n$.
\end{corollary}

\section{The annular semigroup $\mathfrak{A}_n$}\label{annular}
\subsection{Even degree}  This seems to be the most interesting case. It is not possible to apply Proposition \ref{principal proposition} here because $\Sing \mathfrak{A}_n$ is not idempotent generated (nor contained in the type II subsemigroup). We shall not calculate the complexity of $\mathfrak{A}_n$ directly but rather study a certain natural subsemigroup --- the \emph{even annular semigroup} $\mathfrak{EA}_n$  ---, calculate the complexity of the latter and then show that $\mathfrak{A}_n$ has no bigger complexity than $\mathfrak{EA}_n$.

Throughout this subsection let $n$ be even.
Let $\alpha\in \mathfrak{A}_n$ be of rank $r$ and let $a_1<a_2<\cdots<a_r$ and $b'_1<b'_2<\cdots<b'_r$ be the elements of $\dom \alpha$ and $\ran\alpha$, respectively. Then the numbers $a_i$  are alternately even and odd, and likewise are the numbers $b_i$. This is because the nodes strictly between $a_i$ and $a_{i+1}$ as well as strictly between $b'_i$ and $b'_{i+1}$ are entirely involved in inner strings respectively outer strings and hence an even number of nodes must be between $a_i$ and $a_{i+1}$ respectively $b'_i$ and $b'_{i+1}$. A through string $\{i,j'\}$ of $\alpha$ is \emph{even} if $i-j$ is even, and otherwise it is \emph{odd}. Suppose that $\{a_1,b'_{s+1}\}$ is a through string of $\alpha$. Then, by the definition of $\mathfrak{A}_n$, the other through strings of $\alpha$ are exactly the strings $\{a_i,b'_{s+i}\}$ (where the sum $s+i$ has to be taken $\bmod\, r$). It follows that either all through strings of $\alpha$ are even or all are odd. Define the element $\alpha$ to be \emph{even} if every through string of $\alpha$ is even (or equivalently, if $\alpha$ has no odd through string) --- note that the even members of $\mathfrak{A}_n$ coincide with the \emph{oriented diagrams} in \cite{jonesannular}. All diagrams of rank $0$ are even, by definition.  Let $\alpha,\beta\in \mathfrak{A}_n$ and suppose that
 $\mathbf{s}=\underset{k}{\bullet}\!\!\!{\rrbb}\!\!\!\underset{l'}{\bullet}$ is a through string in $\alpha\beta$. By definition of the product in $\mathfrak{A}_n$ there exist a unique number $s\ge 1$ and pairwise distinct  $u_1,v_1,u_2,\dots,v_{s-1},u_s\in [n]$ such that $\mathbf s$ is obtained as the concatenation of the strings
 \begin{equation*}\label{*}
 \underset{k}{\bullet}\!\!\!{\rrbb}\!\!\!\underset{u'_1}{\bullet}\ \underset{u_1}{\bullet}\!\!\!{\rrbb}\!\!\!\underset{v_1}{\bullet}\ \underset{v'_1}{\bullet}\!\!\!{\rrbb}\!\!\!\underset{u'_2}{\bullet}\cdots \underset{u_{s-1}}{\bullet}\!\!\!{\!\!\!\rrbb\!\!}\!\!\!\underset{v_{s-1}}{\bullet}\ \underset{v'_{s-1}}{\bullet}\!\!\!{\!\rrbb\!\!}\!\!\!\underset{u'_s}{\bullet}\ \underset{u_s}{\bullet}\!\!\!{\rrbb}\!\!\!\underset{l'}{\bullet}
 \end{equation*}
 where
 $\mathbf{u}:=\underset{k}{\bullet}\!\!\!{\rrbb}\!\!\!\underset{u'_1}{\bullet}$ is a  through string of $\alpha$, $\mathbf{v}:=\underset{u_s}{\bullet}\!\!\!{\rrbb}\!\!\!\underset{l'}{\bullet}$ is a  through string of $\beta$, all $\underset{u_i}{\bullet}\!\!\!{\rrbb}\!\!\!\underset{v_i}{\bullet}$ are  inner strings of $\beta$ and all $\underset{v'_{i}}{\bullet}\!\!\!{\rrbb}\!\!\!\!\underset{u'_{i+1}}{\bullet}$ are outer strings of $\alpha$. It is easy to see that for each outer string $\{i,j\}$ and each inner string $\{g',h'\}$ of any element $\gamma$ of $\mathfrak{A}_n$ the inequalities $i\not\equiv j\bmod 2$ and $g\not\equiv h\bmod 2$ hold. It follows that $u_i\not\equiv v_i\not\equiv u_{i+1} \bmod 2$ and therefore $u_i\equiv u_{i+1} \bmod 2$ for all $i$ whence $u_1\equiv u_s\bmod 2$. Consequently, $\mathbf{s}$ is even if and only if $\mathbf u$ and $\mathbf v$ are both even or both odd while $\mathbf s$ is odd if and only if exactly one of $\mathbf u$ and $\mathbf v$ is even. In particular, the set $\mathfrak{EA}_n$ of all even members of $\mathfrak{A}_n$ forms a submonoid of $\mathfrak{A}_n$. Moreover, since each projection is even, each idempotent (being the product of two projections) is also even so that $\mathfrak{EA}_n$ contains all idempotents of $\mathfrak{A}_n$. A direct inspection shows that each planar diagram $\alpha\in \mathfrak{J}_n$ is also even, whence $\mathfrak{J}_n$ is a submonoid of $\mathfrak{EA}_n$.

 Similarly as in  $\mathfrak{A}_n$ and $\mathfrak{J}_n$, Green's $\mathrsfs{J}$-relation in $\mathfrak{EA}_n$ is characterized by the rank: two diagrams of $\mathfrak{EA}_n$ are $\mathrsfs{J}$-related if and only if they have the same rank.
The argument is as follows: let $\ep$ and $\eta$ be arbitrary projections of rank
$t$ with $a_1<a_2<\cdots<a_t$  the domain of $\ep$ and $b'_1<b'_2<\cdots<b'_t$ the range of $\eta$; define $\gamma$ to
be the element having the same inner strings as $\ep$, the same outer strings as $\eta$ and the through strings
$\{a_1,b'_1\},\dots,\{a_t,b'_t\}$ in case $a_1\equiv b_1\bmod 2$ while in case $a_1\not\equiv b_1\bmod 2$ the through strings of $\gamma$ can be chosen to be $\{a_1,b'_2\},\{a_2,b'_3\}\dots,\{a_t,b'_1\}$. Then $\gamma\in \mathfrak{EA}_n$, $\ep=\gamma\gamma^*$ and $\eta=\gamma^*\gamma$.

As far as the group of units $\mathfrak{EA}_n^\times$ of $\mathfrak{EA}_n$ is concerned, we see that the diagram
\begin{equation}\label{zeta}
\zeta=\{\{1,2'\},\{2,3'\}\},\dots\{n,1'\}\}\end{equation} is odd and so definitely does not belong to $\mathfrak{EA}_n$. On the other hand,
$$\zeta^2=\{\{1,3'\},\{2,4'\},\dots ,\{n-1,1'\},\{n,2'\}\}$$ is even whence the group of units $\mathfrak{EA}_n^\times$ is cyclic of order $\frac{n}{2}$.
More generally, for each even, positive $r<n$ the maximal subgroups of the $\mathrsfs{J}$-class of all rank-$r$-elements of $\mathfrak{EA}_n$ is cyclic of order $\frac{r}{2}$.

We are going to define two actions $\mathsf{S}$ and $\mathsf{T}$ of $\mathbb{Z}$ on $\mathfrak{A}_n$. The action of $\mathsf{S}$ is by automorphisms, but that of $\mathsf{T}$ is by translations. For $k\in \mathbb{Z}$ let  $\mathsf{S}_k,\mathsf{T}_k:\mathfrak{A}_n\to \mathfrak{A}_n$ be defined as follows: $\alpha\mathsf{S}_k$ is the diagram obtained from $\alpha$ by replacing each string $\{i,j\}$ [respectively $\{i,j'\}$, $\{i',j'\}$] by $\{i+k,j+k\}$ [respectively $\{i+k,(j+k)'\}$, $\{(i+k)',(j+k)'\}$]; $\alpha\mathsf{T}_k$ is obtained from $\alpha$ by replacing each string $\{i,j'\}$ [respectively $\{i',j'\}$] by $\{i,(j+k)'\}$ [respectively $\{(i+k)',(j+k)'\}$]. The addition $+$ has to be taken $\bmod\ n$, of course. We call $\mathsf{S}_k$ the \emph{shift by $k$} and $\mathsf{T}_k$ the \emph{(outer) twist by $k$}. Note that an outer twist leaves unchanged all inner strings. We could similarly define the inner twist by $k$ but we will not need it.

Shifts and twists can be expressed in terms of the unit element $\zeta$ defined in (\ref{zeta}). Namely, for each $\alpha\in \mathfrak{A}_n$ and each $k\in  \mathbb{Z}$ the following hold:
\begin{equation} \label{shifts and twists}
\alpha\mathsf{S}_k=\zeta^{-k}\alpha\zeta^k\mbox{ and }\alpha\mathsf{T}_k=\alpha\zeta^k.\end{equation}
For later use we note that $\alpha\mathsf{S}_k$ is even for every $k$ if and only of $\alpha$ is itself even, and, for every even $k$, $\alpha\mathsf{T}_k$ is even if and only if $\alpha$ itself is even.

In the following we shall show that the singular part of $\mathfrak{EA}_n$ is idempotent generated. In order to simplify notation, we set, for each $i\in [n]$, $\gamma_i:=\gamma_{i,i+1}$, that is, $\gamma_i$ denotes the projection
$$\gamma_i=\{\{i,i+1\},\{i',(i+1)'\}, \{k,k'\}\mid k\ne i,i+1\}.$$
(Addition  has to be taken $\bmod\ n$.) More precisely, we intend to show that $\Sing \mathfrak{EA}_n$ is generated by the projections $\gamma_1,\dots,\gamma_n$. Set $\mathfrak{S}:=\left<\gamma_1,\dots,\gamma_n\right>$; then obviously $\mathfrak{S}\subseteq \Sing \mathfrak{EA}_n$ and $\mathfrak{S}$ is closed under $^*$. Moreover, $\mathfrak{S}$ is closed under  $\mathsf{S}_{\pm 1}$ and therefore closed under $\mathsf{S}_k$ for all $k\in \mathbb{Z}$. This is immediate from the fact that the \textbf{set} $\{\gamma_1,\dots,\gamma_n\}$ is closed under $\mathsf{S}_{\pm 1}$ and that each $\mathsf{S}_k$ is an automorphism. Next, we note that
\begin{equation*}\gamma_{n-1}\cdots\gamma_1=\{\{n-1,n\},\{1',2'\},\{k,(k+2)'\}\mid k=1,\dots n-2\}=:\lambda \end{equation*}
(see Figure \ref{lambda}).

\begin{figure}[ht]
\centering
\begin{tikzpicture}
[scale=0.5]
\foreach \x in {0,3,6,9} \foreach \y in {0,1,2,3,5,6} \filldraw (\x,\y) circle (4pt);
\foreach \x in {12,15,18} \foreach \y in {0,1,3,4,5,6} \filldraw (\x,\y) circle (4pt);
\foreach \x in {0,3,6,9} \draw (\x,4.25) node {$\vdots$};
\foreach \x in {12,15,18} \draw(\x,2.25) node {$\vdots$};
\foreach \y in {0,1,2,3,4,5,6} \draw (10.6,\y) node {$\dots$};
\draw (0,6) -- (9,6); \draw (12,6) -- (15,6);
\draw (0,5) -- (9,5);
\draw (15,4) -- (18,4);
\draw (0,3) -- (6,3); \draw (12,3) -- (18,3);
\draw (0,2) -- (3,2);
\draw (6,1) -- (9,1); \draw (12,1) -- (18,1);
\draw (3,0) -- (9,0); \draw (12,0) -- (18,0);
\draw (19.5,3) node {$=$};
\foreach \y in {0,1,2,3,5,6} \filldraw (21,\y) circle (4pt);
\foreach \y in {0,1,3,4,5,6} \filldraw (24,\y) circle (4pt);
\draw (21,4.25) node {$\vdots$};
\draw (24,2.25) node {$\vdots$};
\draw (21,6) -- (24,4);
\draw (21,5) -- (24,3);
\draw (21,3) -- (24,1);
\draw (21,2) -- (24,0);
\draw (15,6) .. controls (15.5,5.5) and (15.5,5.5) .. (15,5);
\draw (18,6) .. controls (17.5,5.5) and (17.5,5.5) .. (18,5);
\draw (24,6) .. controls (23.5,5.5) and (23.5,5.5) .. (24,5);
\draw (12,5) .. controls (12.5,4.5) and (12.5,4.5) .. (12,4);
\draw (15,5) .. controls (14.5,4.5) and (14.5,4.5) .. (15,4);
\draw (6,3) .. controls (6.5,2.5) and (6.5,2.5) .. (6,2);
\draw (9,3) .. controls (8.5,2.5) and (8.5,2.5) .. (9,2);
\draw (3,2) .. controls (3.5,1.5) and (3.5,1.5) .. (3,1);
\draw (6,2) .. controls (5.5,1.5) and (5.5,1.5) .. (6,1);
\draw (0,1) .. controls (0.5,0.5) and (0.5,0.5) .. (0,0);
\draw (3,1) .. controls (2.5,0.5) and (2.5,0.5) .. (3,0);
\draw (21,1) .. controls (21.5,0.5) and (21.5,0.5) .. (21,0);
\draw (1.5,-1) node {$\gamma_{n-1}$}; \draw (4.5,-1) node {$\gamma_{n-2}$};
\draw (7.5,-1) node {$\gamma_{n-3}$};
\draw (13.5,-1) node {$\gamma_{2}$};
\draw (16.5,-1) node {$\gamma_{1}$};
\draw (22.5,-1) node {$\lambda$};
\draw[densely dotted] (12,4) .. controls (11.5,3.5) and (11.5,3.5) .. (12,3);
\end{tikzpicture}

\caption {The element $\lambda$.} \label{lambda}
\end{figure}
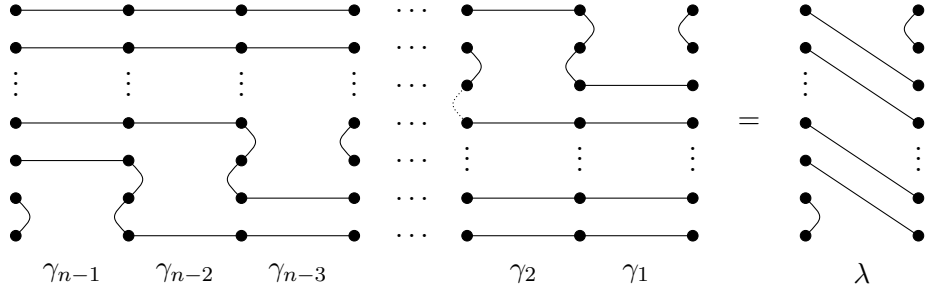\noindent
The element $\lambda$ clearly belongs to $\mathfrak{S}$, and therefore so does each shifted version $\lambda\mathsf{S}_k$ of $\lambda$. Let $\alpha$ be a singular element of $\mathfrak{A}_n$ containing the outer string $\{(n-1)',n'\}$. Then a direct calculation shows that $$\alpha\lambda = \alpha\mathsf{T}_2.$$ More generally, if  $\alpha$ is an arbitrary singular element of $\mathfrak{A}_n$ then there exists $i\in [n]$ such that the outer string $\{(i-1)',i'\}$ belongs to $\alpha$. A similar  calculation then shows that
$$\alpha\cdot\lambda\mathsf{S}_{-(n-i)}=\alpha\mathsf{T}_2.$$
As a consequence, $\mathfrak{S}\mathsf{T}_2\subseteq \mathfrak{S}$, and, since $\mathsf{T}_{-2}=\mathsf{T}_{n-2}$ we infer that  $\mathfrak{S}$ is closed under  twists $\mathsf{T}_k$ for all even $k$. We are ready for a proof of the aforementioned result concerning $\Sing \mathfrak{EA}_n$ and refer to Lemma 2.8 of \cite{jonesannular} for an analogous result in the context of annular algebras. We shall crucially use:
\begin{lemma}[\cite{dosen}, Lemma 2] \label{lemmadosen} $\Sing \mathfrak{J}_n=\left<\gamma_1,\dots,\gamma_{n-1}\right>$.
\end{lemma}
\begin{proposition}\label{singularevenannular} $\Sing \mathfrak{EA}_n=\left<\gamma_1,\dots,\gamma_n\right>$.
\end{proposition}
\begin{proof} Let $\alpha\in \Sing \mathfrak{EA}_n$ and suppose first that $\alpha$ has non-zero rank, that is, $\alpha$ admits a through string. Then there exists some $k\in \mathbb{Z}$ such that $\alpha\mathsf{S}_k$ contains a through string of the form $\{1,j'\}$ for some $j$. Since $\alpha$, and therefore also $\alpha\mathsf{S}_k$ is even, $j$ must be odd.  Then $\alpha\mathsf{S}_k\mathsf{T}_{-(j-1)}$ is still even but contains the through string $\{1,1'\}$, whence $\alpha\mathsf{S}_k\mathsf{T}_{-(j-1)}$ belongs to $\Sing \mathfrak{J}_n$ and therefore, by Lemma \ref{lemmadosen}, $\alpha\mathsf{S}_k\mathsf{T}_{-(j-1)}\in \left<\gamma_1,\dots,\gamma_{n-1}\right>\subseteq \mathfrak{S}$. Consequently,
$$\alpha\in \mathfrak{S}\mathsf{T}_{j-1}\mathsf{S}_{-k}\subseteq \mathfrak{S},$$
as required. Finally, it is easy to see that each rank-zero element of $\mathfrak{A}_n$ belongs actually to $\mathfrak{J}_n$.  Consequently each rank zero element belongs to $\mathfrak{S}$, again as a consequence of Lemma \ref{lemmadosen}.
\end{proof}

For the following considerations let $n\ge 6$ and let the (planar) projection $\varepsilon$ of rank $n-4$ be defined by
\begin{equation}\label{epsilon}
\varepsilon :=\{\{2,3\},\{2',3'\}, \{n-1,n\},\{(n-1)',n'\},\{k,k'\}\mid \\ k\ne 2,3,n-1,n\}.
\end{equation}
The following subsemigroup of $\mathfrak{EA}_n$ will play a crucial role:
\begin{equation*}\label{EA'}
\mathfrak{EA}'_n:=\left<\mathfrak{EA}^\times_n,\gamma_{n-1},\gamma_n\gamma_{n-1},\varepsilon\right>.
\end{equation*}
First we notice that
\begin{equation*}\label{gigi+1gi}
\gamma_{n-1}\gamma_{n}\gamma_{n-1}=\gamma_{n-1}.
\end{equation*} Consequently,
\begin{equation}\label{Lclass-gammai}\gamma_{n-1}\mathrel{\mathrsfs{L}}
\gamma_{n}\gamma_{n-1}= (\gamma_{n}\gamma_{n-1})^2.\end{equation}
Since $\varepsilon=\varepsilon\gamma_{n-1}$  it follows that $\mathfrak{EA}'_n$ is a $\mathcal{T}_1$-semigroup. By Proposition \ref{kernelcomplexity} we obtain:
\begin{corollary}\label{complexity jump} For each even $n\ge 6$ the inequality $c(\mathsf{K}_{\mathbf{G}}(\mathfrak{EA}'_n))<c(\mathfrak{EA}'_n)$ holds.
\end{corollary}
The group of units $\mathfrak{EA}^\times_n$ of $\mathfrak{EA}_n$ consists of the powers of $\zeta^2$, that is,  $$\mathfrak{EA}^\times_n=\{\zeta^2,\zeta^4,\dots,\zeta^{n-2},1\}.$$ Next, we observe that
$$\zeta^{-2}\gamma_{n-1}\zeta^2=\gamma_1,\zeta^{-4}\gamma_{n-1}\zeta^4=\gamma_3,\dots,\zeta^{-(n-2)}\gamma_{n-1}\zeta^{n-2}=\gamma_{n-3}.$$ It follows that
\begin{equation}\label{gai}\gamma_i\in \mathfrak{EA}'_n\mbox{ for each odd $i$.}\end{equation}
 The same argument applied to $\gamma_{n}\gamma_{n-1}$ instead of $\gamma_{n-1}$ implies that \begin{equation}\label{gai-1gai}\gamma_{i+1}\gamma_i\in \mathfrak{EA}'_n\mbox{ for each odd $i$.}\end{equation}
We note that each element of the form $\gamma_{i+1}\gamma_i$ is idempotent.

Next we show that $\mathfrak{EA}_{n-2}$ can be embedded into the idempotent generated subsemigroup $\left<E(\mathfrak{EA}'_n)\right>$ of $\mathfrak{EA}'_n$. First of all, there is an obvious embedding $\mathfrak{EA}_{n-2}\hookrightarrow \mathfrak{EA}_{n}$, namely
$$\alpha\mapsto \alpha\cup\{\{n-1,n\},\{(n-1)',n'\}\}$$
the image of which is exactly the local submonoid $\gamma_{n-1}\mathfrak{EA}_{n}\gamma_{n-1}$. Hence it suffices to show that the latter is contained in $\left<E(\mathfrak{EA}'_n)\right>$. The group of units of $\gamma_{n-1}\mathfrak{EA}_{n}\gamma_{n-1}$ is generated by the diagram
$$\xi:=\{\{1,3'\},\{2,4'\},\dots,\{n-3,1'\},\{n-2,2'\},\{n-1,n\},\{(n-1)',n'\}\}$$
and it is not hard to see that
\begin{equation*}\label{defxi}
\xi=\lambda(\gamma_n\gamma_{n-1})=\gamma_{n-1}(\gamma_{n-2}\gamma_{n-3})\cdots(\gamma_2\gamma_1)(\gamma_n\gamma_{n-1})
\end{equation*} (see Figure \ref{xi}).
\begin{figure}[ht]
\centering
\begin{tikzpicture}
[scale=0.5]
\foreach \x in {0,3,6,9,12,15} \foreach \y in {0,1} \filldraw (\x,\y) circle (4pt);
\filldraw (0,2) circle (4pt); \filldraw (12,2) circle (4pt);
\foreach \x in {3,6,9,15} \draw (\x,2.25) node {$\vdots$};
\foreach \x in {0,3,6,9,12,15} \filldraw (\x,3) circle (4pt);
\draw (0,4.25) node {$\vdots$}; \draw (12,4.25) node {$\vdots$};
\foreach \x in {3,6,9,15} \filldraw (\x, 4) circle (4pt);
\foreach \x in {0,3,6,9,12,15} \foreach \y in {5,6} \filldraw (\x,\y) circle (4pt);
\draw (10.5,3) node {$=$};
\draw (0,6) -- (3,4); \draw (6,6) -- (9,6); \draw (12,6) -- (15,4);
\draw (0,5) -- (3,3); \draw (3,5) -- (9,5); \draw (12,5) -- (15,3);
\draw (3,4) -- (9,4);
\draw (0,3) -- (3,1); \draw (3,3) -- (9,3); \draw (12,3)-- (15,6);
\draw (0,2) -- (3,0); \draw (12,2) -- (15,5);
\draw (3,1) -- (6,1);
\draw (3,6) .. controls (2.5,5.5) and (2.5,5.5) .. (3,5);
\draw (0,1) .. controls (0.5,0.5) and (0.5,0.5) .. (0,0);
\draw (9,1) .. controls (8.5,0.5) and (8.5,0.5) .. (9,0);
\draw (6,1) .. controls (6.5,0.5) and (6.5,0.5) .. (6,0);
\draw (12,1) .. controls (12.5,0.5) and (12.5,0.5) .. (12,0);
\draw (15,1) .. controls (14.5,0.5) and (14.5,0.5) .. (15,0);
\draw (3,6) .. controls (4.5,3) and (4.5,3) .. (3,0);
\draw (6,6) .. controls (4.5,3) and (4.5,3) .. (6,0);
\draw (1.5,-1) node {$\lambda$}; \draw (4.5,-1) node {$\gamma_n$}; \draw (7.5,-1) node {$\gamma_{n-1}$}; \draw (13.5,-1) node{$\xi$};
\end{tikzpicture}
\caption{The element $\xi$.}\label{xi}
\end{figure}
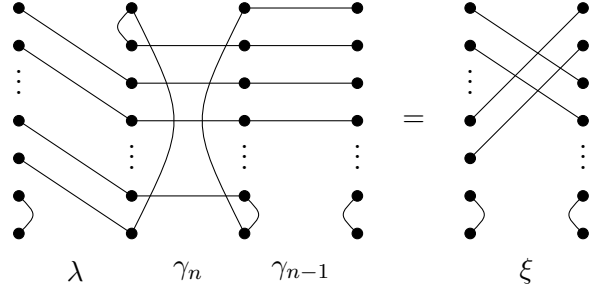
Therefore, the element $\xi$ belongs to $\left<E(\mathfrak{EA}'_n)\right>$
by  (\ref{gai-1gai}) and since each $\gamma_{i+1}\gamma_i$ is idempotent. Consequently, the group of units of $\gamma_{n-1}\mathfrak{EA}_n\gamma_{n-1}$ is contained in $\left<E(\mathfrak{EA}'_n)\right>$.  It remains to show that the singular part of $\gamma_{n-1}\mathfrak{EA}_{n}\gamma_{n-1}$ is contained in $\left<E(\mathfrak{EA}'_n)\right>$. First of all, the projections
$$\gamma'_1:=\gamma_1\gamma_{n-1},\gamma'_3:=\gamma_3\gamma_{n-1},\dots,\gamma'_{n-3}:=\gamma_{n-3}\gamma_{n-1}$$
are all contained in $E(\mathfrak{EA}'_n)$ by (\ref{gai}). We note that $\xi^{\frac{n-2}{2}}=\gamma_{n-1}$ (the identity of the local monoid) and therefore $\xi^*=\xi^{\frac{n-4}{2}}$. The latter element belongs to  $\left<E(\mathfrak{EA}'_n)\right>$ since $\xi$ does so. It follows that the projections
$$\gamma'_2:=\varepsilon, \gamma'_4:=\xi^*\varepsilon\xi,\dots,\gamma'_{n-2}:=(\xi^*)^{\frac{n-4}{2}}\varepsilon\xi^{\frac{n-4}{2}}$$ are also contained in $\left<E(\mathfrak{EA}'_n)\right>$. But by Proposition \ref{singularevenannular}, applied to $\mathfrak{EA}_{n-2}\cong\gamma_{n-1}\mathfrak{EA}_n\gamma_{n-1}$ the singular part of that monoid is generated by the $n-2$ projections $\gamma'_1,\dots,\gamma'_{n-2}$. We have thus proved the following:
\begin{lemma}\label{inclusioninE} For each even $n\ge 6$, $\mathfrak{EA}_{n-2}$ is isomorphic to a subsemigroup of $\left<E(\mathfrak{EA}'_n)\right>$.
\end{lemma}
In combination with Corollary \ref{complexity jump} we are able to formulate the next (crucial) statement.
\begin{proposition}\label{complexity expected} For each even $n\ge 6$ the inequality $c(\mathfrak{EA}_{n-2})<c(\mathfrak{EA}_n)$ holds.
\end{proposition}
\begin{proof}This follows from
\begin{align*}c(\mathfrak{EA}_{n-2})&\le c(\left<E(\mathfrak{EA}'_n)\right>)&\mbox{ by Lemma \ref{inclusioninE} }\\
                                    &\le c(\mathsf{K}_{\mathbf{G}}(\mathfrak{EA}'_n))&\\
                                    &< c(\mathfrak{EA}'_n)&\mbox{ by Corollary \ref{complexity jump} }\\
                                     &\le c(\mathfrak{EA}_n).&\\
\end{align*}
\end{proof}
It is straightforward that $c(\mathfrak{EA}_2)=0$; in $\mathfrak{EA}_4$ the only essential $\mathrsfs{J}$-class  is the set of all rank $4$ elements hence $c(\mathfrak{EA}_4)=1$. For each $n\in \mathbb{N}$ we have that $c(\mathfrak{EA}_{2n})\le c(\mathfrak{EA}_{2n-2})+1$. Indeed, $\mathfrak{EA}_{2n-2}\cong\gamma_{n-1}\mathfrak{EA}_{2n}\gamma_{n-1}$, whence $c(\mathfrak{EA}_{2n-2})=c(\mathfrak{EA}_n\gamma_{n-1}\mathfrak{EA}_n)$ by Lemma \ref{submonoids} and the latter semigroup is the ideal $\Sing \mathfrak{EA}_n$ of all singular elements of $\mathfrak{EA}_n$; the claim then follows from Lemma \ref{ideal lemma} by taking into account that $c(\mathfrak{EA}_n/\Sing \mathfrak{EA}_n)=1$. This, in combination with
Proposition \ref{complexity expected}, then gives $c(\mathfrak{EA}_{2n})=c(\mathfrak{EA}_{2n-2})+1$ for all $n \ge 2$. By induction we get:
\begin{theorem}\label{cEAn}  The equality $c(\mathfrak{EA}_{2n})=n-1$ holds for each positive integer $n$.
\end{theorem}
An immediate consequence is that $c(\mathfrak{A}_{2n})\ge n-1$ for all positive integers. On the other hand, since the depth of $\mathfrak{A}_{2n}$ is $n$  we also have $c(\mathfrak{A}_{2n})\le n$ for all $n$. In order to determine the exact value we need to look at small values of $n$. Clearly, $c(\mathfrak{A}_2)=1$ since $\mathfrak{A}_2$ has exactly one essential $\mathrsfs{J}$-class. It turns out that the crucial point is the value $c(\mathfrak{A}_4)$.  Although $\mathfrak{A}_4$ has two essential $\mathrsfs{J}$-classes its complexity is only $1$.
\begin{lemma}\label{degree4} $c(\mathfrak{A}_4)=1$.
\end{lemma}
\begin{proof} By Proposition \ref{A*G} it suffices to show that the type II subsemigroup $\mathsf{K}_{\mathbf{G}}(\mathfrak{A}_4)$ is aperiodic. Define relations $\tau_1:\mathfrak{A}_4\to \mathfrak{A}^\times_4$ and $\tau_2:\mathfrak{A}_4\to \{-1,1\}$ as follows:
\begin{equation*}
x\tau_1=
\begin{cases}
x &\text{if $x\in \mathfrak{A}_4^\times$,}\\
\mathfrak{A}^\times_4 &\text{if $x\notin \mathfrak{A}_4^\times$}
\end{cases}
\end{equation*}
and
\begin{equation*}
x\tau_2=
\begin{cases}
-1 &\text{if $\rk x\ge 2$ and $x$ is odd,}\\
1 &\text{if $\rk x\ge 2$ and $x$ is even,}\\
\{-1,1\} &\text{if $\rk x=0$.}
\end{cases}
\end{equation*}
It is easily checked that $\tau_1$ and $\tau_2$ are relational morphisms. Let $\tau=\tau_1\times \tau_2$; then $1\tau^{-1}=\mathsf{Sing}\mathfrak{EA}_4\cup \{1\}$. Since $\mathsf{Sing}\mathfrak{EA}_4$ is idempotent generated we infer that  $\mathsf{K}_{\mathbf G}(\mathfrak{A}_4)=\mathsf{Sing}\mathfrak{EA}_4\cup \{1\}$ and the latter is aperiodic.
\end{proof}
It is worth to point out that the preceding Lemma is also a consequence of Tilson's $2\mathrsfs{J}$-class Theorem \cite[Theorem 4.15.2]{qtheory}. As in the case of the even annular semigroup, we have $c(\mathfrak{A}_{2n})\le c(\mathfrak{A}_{2n-2})+1$ (the argument is very much analogous to the one before the statement of Theorem \ref{cEAn}). This, in combination with $c(\mathfrak{A}_4)=1$ and $c(\mathfrak{A}_{2n})\ge n-1$ for all $n$ then leads to the main result in this subsection.
\begin{theorem}\label{annulareven}  The equality $c(\mathfrak{A}_{2n})=n-1$ holds for each  integer $n\ge 2$; for $n=1$ the equality $c(\mathfrak{A}_2)=1$ holds.
\end{theorem}
\subsection{Odd degree} Throughout this subsection let $n$ be odd. This case is easier since the singular part $\Sing \mathfrak{A}_n$ is idempotent generated. An analogous statement in the context of annular algebras has been mentioned without proof in \cite[Remark 2.9]{jonesannular}. We retain the notation of the preceding subsection.
\begin{proposition} For each odd positive integer $n$, $\Sing \mathfrak{A}_n=\left<\gamma_1,\dots,\gamma_n\right>$.
\end{proposition}
\begin{proof}
The proof is similar to that of Proposition \ref{singularevenannular}; let $\mathfrak{S}:=\left<\gamma_1,\dots,\gamma_n\right>$. Once again, $\mathfrak{S}$ is closed under all shifts $\mathsf{S}_k$. As in the case for even $n$, for
$$\xi=\gamma_{n-1}\gamma_{n-2}\cdots\gamma_1\gamma_n\gamma_{n-1}$$ we obtain
$$\xi=\{\{1,3'\},\{2,4'\},\dots, \{n-3,1'\},\{n-2,2'\},\{n-1,n\},\{(n-1)',n'\}\}$$ (see Figure \ref{xi}). Thus $\xi\vert_{[n-2]}$ realizes the cyclic permutation on $[n-2]$ given by $x\mapsto x+2\ (\bmod\ n-2)$. Since $n-2$ is odd, this permutation has order $n-2$ and so $\xi$ generates the group $\mathrsfs{H}$-class of $\gamma_{n-1}$. More specifically,
$$\xi^{\frac{n-1}{2}}=\{\{1,2'\},\{2,3'\}\dots,\{n-2,1'\},\{n-1,n\},\{(n-1)',n'\}\}.$$
It follows that
$$\tau:=\xi^{\frac{n-1}{2}}\gamma_n=\{\{1,2'\},\{2,3'\}\dots,\{n-2,(n-1)'\},\{n-1,n\},\{1',n'\}\}$$
(see Figure \ref{xigan}), and $\tau$ belongs to $\mathfrak{S}$.
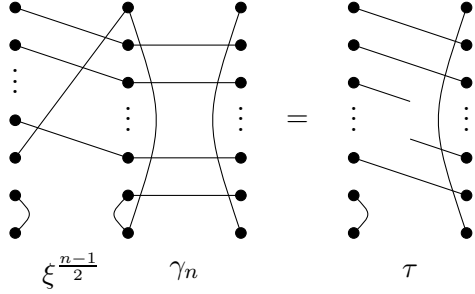
\begin{figure}
\centering
\begin{tikzpicture}
[scale=0.5]
\foreach \x in {0,3,6} \foreach \y in {0,1,2} \filldraw (\x,\y) circle (4pt);
\foreach \x in {0,3,6} \foreach \y in {5,6} \filldraw (\x,\y) circle (4pt);
\filldraw (0,3) circle (4pt); \filldraw (3,4) circle (4pt); \filldraw (6,4) circle (4pt);
\draw (3,3.25) node{$\vdots$}; \draw (6,3.25) node{$\vdots$}; \draw (0,4.25) node{$\vdots$};
\draw (0,6) -- (3,5) -- (6,5);
\draw (0,5) -- (3,4) -- (6,4);
\draw (0,3) -- (3,2) -- (6,2);
\draw (3,1) -- (6,1); \draw ((0,2) --(3,6);
\draw (0,1) .. controls (0.5,0.5) and (0.5,0.5) .. (0,0);
\draw (3,1) .. controls (2.5,0.5) and (2.5,0.5) .. (3,0);
\draw (6,0) .. controls (5,3) and (5,3) .. (6,6);
\draw (3,0) .. controls (4,3) and (4,3) .. (3,6);
\draw (1.5,-1) node{$\xi^{\frac{n-1}{2}}$};
\draw (4.5,-1) node{$\gamma_n$};
\draw (7.5,3) node{$=$};
\foreach \x in {9,12} \foreach \y in {0,1,2,4,5,6} \filldraw (\x,\y) circle (4pt);
\foreach \x in {9,12} \draw (\x,3.25) node{$\vdots$};
\draw (9,6) -- (12,5);
\draw (9,5) -- (12,4);
\draw (9,4) -- (10.5,3.5);
\draw (10.5,2.5) -- (12,2);
\draw (9,2) -- (12,1);
\draw (9,1) .. controls (9.5,0.5) and (9.5,0.5) .. (9,0);
\draw (12,0) .. controls (11,3) and (11,3) .. (12,6);
\draw (10.5,-1) node{$\tau$};
\end{tikzpicture}
\caption{The element $\tau$.}\label{xigan}
\end{figure}
Suppose now that $\alpha$ is a singular element of $\mathfrak{A}_n$ containing the outer string $\{(n-1)',n'\}$. Then $\alpha\tau= \alpha\mathsf{T}_1$. More generally, if $\alpha$ is an arbitrary singular element of $\mathfrak{A}_n$ then it contains the outer string $\{(i-1)',i'\}$ for some $i$. A direct calculation shows that
$$\alpha\cdot\tau\mathsf{S}_{-(n-i)}=\alpha\mathsf{T}_1.$$ As a consequence, $\mathfrak{S}\mathsf{T}_1\subseteq \mathfrak{S}$. Since $\mathsf{T}_{-1}=\mathsf{T}_{n-1}$, $\mathfrak{S}$ is also closed under $\mathsf{T}_{-1}$. Altogether, $\mathfrak{S}$ is closed under twists $\mathsf{T}_k$ for \textbf{each} $k\in \mathbb{Z}$. Let now $\alpha$ be an arbitrary  element of $\Sing \mathfrak{A}_n$. Then there exist integers $k,\ell$ such that $\alpha\mathsf{S}_k\mathsf{T}_\ell$ contains the through string $\{1,1'\}$. But then $\alpha\mathsf{S}_k\mathsf{T}_\ell$ is planar and by Lemma  \ref{lemmadosen} $\alpha\mathsf{S}_k\mathsf{T}_\ell\in \left<\gamma_1,\dots,\gamma_{n-1}\right>\subseteq\mathfrak{S}$ so that
$$\alpha\in \mathfrak{S}\mathsf{T}_{-\ell}\mathsf{S}_{-k}\subseteq \mathfrak{S}.$$ Altogether we have obtained the inclusion $\Sing \mathfrak{A}_n\subseteq \mathfrak{S}$.
\end{proof}
Since each $\gamma_i$ is contained in $\left<\mathfrak{A}^\times_n,\gamma_{n-1}\right>$ it also follows that $\mathfrak{A}_n=\left<\mathfrak{A}^\times_n,\gamma_{n-1}\right>$ and we may apply Proposition \ref{principal proposition}.
\begin{proposition}  The equality $c(\mathfrak{A}_{2n+1})=c(\mathfrak{A}_{2n-1})+1$ holds for each positive integer $n$.
\end{proposition}
\begin{proof} We apply Proposition \ref{principal proposition} to $S=\mathfrak{A}_{2n+1}$ and $e=\gamma_{2n}$. Since $\mathfrak{A}_{2n-1}\cong \gamma_{2n}\mathfrak{A}_{2n+1}\gamma_{2n}$ we obtain $c(\mathfrak{A}_{2n+1})=c(\mathfrak{A}_{2n-1})+1$.
\end{proof}
Since $c(\mathfrak{A}_1)=0$ we get by induction:
\begin{theorem}\label{annularodd}  The equality $c(\mathfrak{A}_{2n-1})=n-1$ holds for each positive integer $n$.
\end{theorem}
\subsection{The partial annular semigroup $P\mathfrak{A}_n$} We are going to treat the partial version $P\mathfrak{A}_n$ of $\mathfrak{A}_n$.  First of all we clearly have
\begin{equation}\label{partialfull}c(P\mathfrak{A}_n)\ge c(\mathfrak{A}_n)\mbox{ for each positive integer $n$.} \end{equation}
The next arguments are analogous to the corresponding ones in the context of the Brauer semigroups.
 Let $P\mathfrak{A}_n^{(n-2)}$ be the ideal of $P\mathfrak{A}_n$ consisting of all elements of rank at most $n-2$. The Rees quotient $P\mathfrak{A}_n/P\mathfrak{A}_n^{(n-2)}$ is an inverse semigroup whence its complexity is $1$. The Ideal Theorem then implies
$$c(P\mathfrak{A}_n)\le c(P\mathfrak{A}_n^{(n-2)})+1.$$
Moreover, since $P\mathfrak{A}_n^{(n-2)}= P\mathfrak{A}_n\gamma_1 P\mathfrak{A}_n$ and $P\mathfrak{A}_{n-2}\cong \gamma_1 P\mathfrak{A}_n\gamma_1$ Lemma \ref{submonoids} implies
\begin{equation}\label{estimationpartial} c(P\mathfrak{A}_n)\le c(P\mathfrak{A}_{n-2})+1\end{equation}
for all $n\ge 2$. Since $c(P\mathfrak{A}_1)=0$, in combination with (\ref{partialfull}) and Theorem \ref{annularodd} this yields:
\begin{theorem} The equality $c(P\mathfrak{A}_{2n-1})=c(\mathfrak{A}_{2n-1})=n-1$ holds for each positive integer $n$.
\end{theorem}

The even case is again more difficult. The results obtained so far imply that $n-1\le c(P\mathfrak{A}_{2n})\le n$ for all $n$. The complexity of $P\mathfrak{A}_2$ is of course equal to $1$. The author does not know whether $c(P\mathfrak{A}_4)$ equals $1$ or $2$. The same argument as for $\mathfrak{A}_4$ in order to show that $c(P\mathfrak{A}_4)=1$ cannot be applied since the type II subsemigroup $\mathsf{K}_{\mathbf{G}}(P\mathfrak{A}_4)$ is not aperiodic. In particular, $P\mathfrak{A}_4$ is not contained in $\mathbf{A}*\mathbf{G}$. It is easy to see that $P\mathfrak{A}_4$ is neither contained in $\mathbf{G}*\mathbf{A}$ (not even in $L\mathbf{G}\malc\mathbf{A}$).

That $\mathsf{K}:=\mathsf{K}_\mathbf{G}(P\mathfrak{A}_4)$ is not aperiodic can be seen as follows. The even diagrams \tikz[scale=0.1]{\draw (0,3) -- (3,3); \draw (0,2) -- (3,0); \draw (0,1) .. controls (1,0.5) and (1,0.5) .. (0,0);
\draw (3,2) .. controls (2,1.5) and (2,1.5) .. (3,1);} and \tikz[scale=0.1]{\draw (0,0) -- (3,2); \draw (0,3) -- (3,1); \draw (0,1) .. controls (1,1.5) and (1,1.5) .. (0,2);
\draw (3,0) .. controls (1.5,1.5) and (1.5,1.5) .. (3,3);} belong to $\mathsf{K}$ (since they are in $\Sing \mathfrak{EA_4}$ and so, by Proposition \ref{singularevenannular} in $\left<E(\mathfrak{A}_4)\right>$).  Conjugation of the former element by
\tikz[scale=0.1]{ \draw (0,3) -- (3,3); \draw (0,0) -- (3,0); \filldraw (0,2) circle (2pt); \filldraw (3,1) circle (2pt);\draw (0,1) -- (3,2);} shows that \tikz[scale=0.1]{ \draw (3,3) -- (0,3); \draw (0,1) -- (3,0); \filldraw (0,0) circle (2pt);\filldraw (0,2)circle (2pt); \filldraw (3,1) circle (2pt);\filldraw  (3,2) circle (2pt);} is in $\mathsf{K}$. Since $\mathsf{K}$ is closed under shifts, \tikz[scale=0.1]{ \draw (3,2) -- (0,3); \draw (0,1) -- (3,1); \filldraw (0,0) circle (2pt);\filldraw (0,2)circle (2pt); \filldraw (3,0) circle (2pt);\filldraw  (3,3) circle (2pt);}
is in $\mathsf{K}$. For symmetry reasons, also \tikz[scale=0.1]{ \draw (0,1) -- (3,1); \draw (0,2) -- (3,3); \filldraw (0,0) circle (2pt);\filldraw (0,3)circle (2pt); \filldraw (3,0) circle (2pt);\filldraw  (3,2) circle (2pt);} is in $\mathsf{K}$. The product \tikz[scale=0.1]{ \draw (0,1) -- (3,1); \draw (0,2) -- (3,3); \filldraw (0,0) circle (2pt);\filldraw (0,3)circle (2pt); \filldraw (3,0) circle (2pt);\filldraw  (3,2) circle (2pt);}\tikz[scale=0.1]{ \draw (3,3) -- (0,3); \draw (0,1) -- (3,0); \filldraw (0,0) circle (2pt);\filldraw (0,2)circle (2pt); \filldraw (3,1) circle (2pt);\filldraw  (3,2) circle (2pt);}\tikz[scale=0.1]{ \draw (0,3) -- (3,1); \draw (0,0) -- (3,2); \filldraw (0,1) circle (2pt);\filldraw (0,2)circle (2pt); \filldraw (3,0) circle (2pt);\filldraw  (3,3) circle (2pt); \draw (0,1) .. controls (1,1.5) and (1,1.5) .. (0,2);
\draw (3,0) .. controls (1.5,1.5) and (1.5,1.5) .. (3,3);} is equal to \tikz[scale=0.1]{ \draw (0,1) -- (3,2); \draw (0,2) -- (3,1); \filldraw (0,0) circle (2pt);\filldraw (0,3)circle (2pt); \filldraw (3,0) circle (2pt);\filldraw  (3,3) circle (2pt);\draw (3,0) .. controls (2,1.5) and (2,1.5) .. (3,3);} which is a (= the) non-idempotent member of the group $\mathrsfs{H}$-class of the idempotent  \tikz[scale=0.1]{ \draw (0,1) -- (3,1); \draw (0,2) -- (3,2); \filldraw (0,0) circle (2pt);\filldraw (0,3)circle (2pt); \filldraw (3,0) circle (2pt);\filldraw  (3,3) circle (2pt);\draw (3,0) .. controls (2,1.5) and (2,1.5) .. (3,3);} and is contained in $\mathsf{K}$.

Since $P\mathfrak{A}_4$ has three essential $\mathrsfs{J}$-classes, Tilson's Theorem \cite[Theorem 4.15.2]{qtheory} cannot be applied to compute $c(P\mathfrak{A}_4)$. However, it can be checked that each divisor of $P\mathfrak{A}_4$ which has at most $2$ essential $\mathrsfs{J}$-classes has complexity at most $1$. It should be clear from the discussion in the present section that if $c(P\mathfrak{A}_4)=1$ happened to hold then we immediately would know that $c(P\mathfrak{A}_{2n})=n-1=c(\mathfrak{A}_{2n})$ for all $n\ge 2$, while, if $c(P\mathfrak{A}_4)=2$ were true then the we could not draw any conclusion about the value of $c(P\mathfrak{A}_{2n})$ other than $n-1\le c(P\mathfrak{A}_{2n})\le n$ for all $n\ge 2$ (though it is very likely that in the latter case $c(P\mathfrak{A}_{2n})=n$ holds for all $n\ge 1$).


\begin{thebibliography}{99}

\bibitem{almeidauniv}     
     J.~Almeida, \emph{Finite Semigroups and Universal Algebra},
     World Scientific, Singapore, {1994}.
\bibitem{ash} C.~Ash, \emph{Inevitable graphs: a proof of the type II conjecture and related decision procedures}, Int. J. Algebra Comput. \textbf{1} (1991), 127--146.
\bibitem{aubrauer} K.~Auinger, \emph{Pseudovarieties generated by Brauer type monoids}, Forum Math. (to appear), DOI: 10.1515/form.2011.146.
\bibitem{adv2} K.~Auinger, I.~Dolinka, M.~V.~Volkov, \emph{Equational theories of semigroups with involution}, J.~Algebra \textbf{369}  (2012), 203--225.
\bibitem{dosen} M.~Borisavljevic, K.~Do\v sen, Z.~Petri\'c, \emph{Kauffman monoids}, J.~Knot Theory Ramifications \textbf{11} (2002), 127--143.
\bibitem{brauer} R.~Brauer, \emph{On algebras which are connected with the semisimple continuous groups}, Ann. Math. 38 (1937),  857--872.
\bibitem{grahamlehrer}
J.~J.~Graham, G.~I.~Lehrer, \emph{Cellular algebras}, Invent. Math. \textbf{123} (1996), 1--34.
\bibitem{halversonram} T.~Halverson, A.~Ram, \emph{Partition algebras}, Europ.~J.~Combinatorics \textbf{26} (2005), 869--921.
\bibitem{jonesannular} V.~R.~F.~Jones, \emph{A quotient of the affine Hecke algebra in the Brauer algebra}, L'Enseign.~Math. \textbf{40} (1994), 313--344.
\bibitem{kambites} M.~Kambites, \emph{On the Krohn--Rhodes complexity of semigroups of upper triangular matrices}, Int. J. Algebra Comput. \textbf{17} (2007), 187--201.
\bibitem{krohnrhodes} K.~Krohn, J.~Rhodes, \emph{Algebraic theory of machines I. Prime decomposition theorem for finite semigroups and machines}, Trans. Amer. Math. Soc. \textbf{116} (1965), 450--464.
\bibitem{kudmaz} G.~Kudryavtseva, V.~Mazorchuk, \emph{On presentation of Brauer-type monoids}, Cent.~Europ.~J.~Math. \textbf{4} (2006), 413--434.
\bibitem{FitzGerald}
K.~W.~Lau, D.~G.~Fitzgerald, \emph{Ideal structure of the Kauffman and related monoids}, Comm. Algebra \textbf{34} (2006), 2617--2629.
\bibitem{malmaz} V.~Maltcev, V.~Mazorchuk, \emph{Presentation of the singular part of the Brauer monoid}, Math.~Bohem. \textbf{132} (2007), 297--323.
\bibitem{Maz1}
V.~Mazorchuk, \emph{On the structure of Brauer semigroups and its partial analogue}, Problems in Algebra \textbf{13}  (1998), 29--45.
\bibitem{Maz2}
V.~Mazorchuk, \emph{Endomorphisms of $\mathfrak{B}_n$, ${\mathcal P}\mathfrak{B}_n$ and $\mathfrak{C}_n$}, Comm. Algebra \textbf{30} (2002), 3489--3513.
\bibitem{rhodesTn} J.~Rhodes, \emph{Some results on finite semigroups}, J. Algebra \textbf{4} (1966), 471--504.
\bibitem{rhodesBn} J.~Rhodes, \emph{Finite binary relations have no more complexity than finite functions}, Semigroup Forum \textbf{7} (1974), 92--103.
\bibitem{qtheory}
     J.~Rhodes, B.~Steinberg, \emph{The $\mathfrak{q}$-theory of
     Finite Semigroups}, Springer, New York, 2009.
\bibitem{temperleylieb}
H.~N.~V.~Temperley, E.~H.~Lieb, \emph{Relations between the `percolation' and `coloring' problem and other graph-theoretic
problems associated with regular planar lattices: some exact results for the `percolation' problem}, Proc. Roy. Soc.
London Ser.~A  \textbf{322}, (1971) 251--280.
\bibitem{Wil}
S. Wilcox, Cellularity of diagram algebras as twisted semigroup algebras, J. Algebra \textbf{309}, (2007) 10--31.
\end{thebibliography}
\end{document}